\RequirePackage{fix-cm}
\documentclass[smallextended]{svjour3}
\usepackage{amssymb,latexsym}       
\smartqed  
\usepackage{graphicx}
\begin{document}

\title{On the typical rank of real polynomials (or symmetric tensors) with a fixed border rank\thanks{The author was partially supported by MIUR and GNSAGA of INdAM (Italy).}
}


\titlerunning{typical rank}       

\author{Edoardo Ballico}

\authorrunning{E. Ballico}

\institute{E. Ballico\at
            Dept. of Mathematics, University of Trento, 38123 Trento (Italy) \\
              Tel.: 39+0461281646\\
              Fax: 39+0461281624\\
              \email{ballico@science.unitn.it}}

\date{Received: date / Accepted: date}

\maketitle

\begin{abstract}
Let $\sigma _b(X_{m,d}(\mathbb {C}))(\mathbb {R})$, $b(m+1) < {m+d \choose m}$, denote the set of all degree $d$ real homogeneous polynomials in $m+1$ variables (i.e. real symmetric tensors
of format $(m+1)\times \cdots \times (m+1)$, $d$ times) which
have border rank $b$ over $\mathbb {C}$. It has a partition into manifolds of real dimension $\le b(m+1)-1$ in which the real rank is constant. A typical rank
of $\sigma _b(X_{m,d}(\mathbb {C}))(\mathbb {R})$ is a rank associated to an open part of dimension $b(m+1)-1$.
Here we classify all typical ranks when $b\le 7$
and $d, m$ are not too small. For a larger sets of $(m,d,b)$ we prove that $b$ and $b+d-2$ are the two first typical ranks. In the case $m=1$ (real bivariate polynomials) we prove that $d$ (the maximal possible a priori value of the real
rank) is a typical rank for every $b$.
\end{abstract}

\section{Introduction}
\label{intro}
Fix an integer $m>0$ and call $\mathbb {K}$ either the real field $\mathbb {R}$ or the complex field $\mathbb {C}$. For every integer $d\ge 0$ let
$\mathbb {K}[x_0,\dots ,x_m]_d$ denote the $\mathbb {K}$-vector space of all degree $d$ polynomials in the variables $x_0,\dots ,x_m$ and with coefficients
in $\mathbb {K}$. Now assume $d>0$ and fix $f\in \mathbb {K}[x_0,\dots ,x_m]_d\setminus \{0\}$. The rank or the symmetric tensor rank $r_{\mathbb {K}}(f)$ of $f$ with respect
to $\mathbb {K}$ is the minimal integer $s>0$ such that
\begin{equation}\label{eqi1}
f = c_1\ell _1^d+\cdots +c_s\ell _s^d
\end{equation}
for some $\ell _i\in \mathbb {K}[x_0,\dots ,x_m]_1$ and some $c_i\in \mathbb {K}$ (\cite{l}, \S 5.4). If either $\mathbb {K}=\mathbb {C}$ or  $d$ is odd and $\mathbb {K}=\mathbb {R}$,
then we may take $c_i=1$ for all $i$ without any loss of generality. If $d$ is even and $\mathbb {K}=\mathbb {R}$, then we may take $c_i\in \{-1,1\}$ without any loss
of generality. We may see any symmetric tensor of format $(m+1)\times \cdots \times (m+1)$ ($d$ products) as a polynomial $f\in \mathbb {K}[x_0,\dots ,x_m]_d$. Any tensor
has a tensor rank, but it is not known if for a symmetric tensor its rank as a tensor and its rank as a polynomial are the same. Comon's conjecture asks if symmetric rank is equal to rank for all symmetric tensors. In this paper
we will always consider the symmetric rank of a symmetric tensor $T$, i.e. we will see $T$ as as a polynomial $f$ and take $r_{\mathbb {K}}(f)$ as the integer associated to $T$;  we call it the symmetric tensor rank of $T$ or $f$. These notions appears in several different topics in engineering
(\cite{bk}, \cite{cc}, \cite{cm}, \cite{ctdc}, \cite{cglm}, \cite{k}, \cite{ls}, \cite{mh}, \cite{sbg}, \cite{sbg1}, \cite{t}). The book \cite{l} contains a huge bibliography which contains both
the applied side and the theoretical side of this topic. In many applications the data are
known only approximatively. In this case over $\mathbb {C}$ there is a non-empty and dense open subset $\mathcal {U}$ of $ \mathbb {C}[x_0,\dots ,x_m]_d \cong \mathbb {C}^{r+1}$,
$r:= {m+d \choose m} -1$, such that all $f\in \mathcal {U}$ have the same rank $r_{\mathbb {C}}(f)$ and this rank is called the generic rank. If $m=1$ and $d\ge 2$,
then $ \lfloor (d+2)/2\rfloor$ is the generic rank (\cite{bgi}, \cite{co}, \cite{l}, \cite{lt}). In the case $m \ge 2$ the generic rank is also known by a theorem of Alexander and Hirschowitz: for each $d\ge 3$
the generic rank is $\lceil {m+d \choose m}/(m+1)\rceil$ except in $4$ well-studied exceptional cases (\cite{ah}, \cite{ah1}, \cite{bo}, \cite{c}). The situation in the case $\mathbb {K}=\mathbb {R}$ is more complicated, because $ \mathbb {R}[x_0,\dots ,x_m]_d \cong \mathbb {R}^{r+1}$ has several non-empty open subsets for the euclidean topology in which the real rank is constant, but these constants are not the same for different open sets. The ranks with respect to $\mathbb {R}$ arising in these open
subsets are called the typical ranks (\cite{co}, \cite{cr}). Only the bivariate cases was recently solved (\cite{bl}). However, quite often we get polynomials (or symmetric tensors) with some
constraints and one may study the generic rank (case $\mathbb {K}=\mathbb {C}$) or the typical ranks (case $\mathbb {K}=\mathbb {R}$) for polynomials with those constraints.
An algebraic constraint which is quite studied over $\mathbb {C}$ is the border rank (\cite{l}, Chapter 5), which we now define. We always assume $d\ge 2$. Let
$\nu _d: \mathbb {P}^m(\mathbb {K}) \to \mathbb {P}^r(\mathbb {K})$, $r:= {m+d \choose m}-1$, be the order $d$ Veronese embedding, i.e.
the embedding induced by the vector space $\mathbb {K}[x_0,\dots ,x_m]_d$. Set $X_{m,d}(\mathbb {K}):= \nu _d(\mathbb {P}^m(\mathbb {K}))$. Let $Y\subset \mathbb {P}^n(\mathbb {K})$ be
any set spanning $\mathbb {P}^n(\mathbb {K})$. For any $P\in \mathbb {P}^n(\mathbb {K})$ the $Y$-rank $r_Y(P)$ or $r_{Y,\mathbb {K}}(P)$ of $P$ with respect
to $\mathbb {K}$ is the minimal cardinality
of a set $S\subset Y$ such that $P\in \langle S\rangle$, where $\langle \ \ \rangle$ denote the linear span. Any $f\in \mathbb {K}[x_0,\dots ,x_m]_d\setminus \{0\}$
induces $P\in \mathbb {P}^r(\mathbb {K})$ and $r_{X_{m,d}(\mathbb {K})}(P) = r_{\mathbb {K}}(f)$. Now assume that $Y\subset \mathbb {P}^n$ is a geometrically integral variety defined over $\mathbb {K}$. For each integer $b >0$ let
$\sigma _b(Y(\mathbb {C}))$ denote the closure of the union of all linear spaces $\langle S\rangle$ with $S\subset Y(\mathbb {C})$ and $\sharp (S) =b$ (or $\le b$).  For any $P\in \mathbb {P}^r(\mathbb {C})$ the border rank $br (P)$ of $P$ is the minimal integer $b\ge 1$ such that $P\in \sigma _b(X_{m,d}(\mathbb {C}))$. The set $\sigma _b(X_{m,d}(\mathbb {C}))$ is an integral variety
defined over $\mathbb {R}$ and we may look at its real points $\sigma _b(X_{m,d}(\mathbb {C}))(\mathbb {R})$. The integer $\alpha (m,d,b):= \dim (\sigma _b(X_{m,d}(\mathbb {C})) )$ is known
(in the case $m \ge 2$ by the quoted theorem of Alexander-Hirschowitz, while it was classically known that
$\alpha (1,d,b) = \min \{d,2b-1\}$ for all $d, b$ (\cite{bgi}, \cite{lt}, \cite{l}, Chapter 5). The set $\sigma _b(X_{m,d}(\mathbb {C}))(\mathbb {R})$ has a partition into topological manifolds of dimensions $\le \alpha (m,d,b)$
with finitely many connected open subsets of dimension $\alpha (m,d,b)$ with the additional condition that on each of these connected pieces the symmetric rank is constant.
We call any such symmetric rank a typical $b$-rank or a typical rank for the border rank $b$. Notice that we use  $\sigma _b(X_{m,d}(\mathbb {C}))(\mathbb {R})$, not
something constructed only using $X_{m,d}(\mathbb {R})$. In many cases we have the equations of $\sigma _b(X_{m,d}(\mathbb {C}))$ and hence we may check
if $P\in (\sigma _b(X_{m,d}(\mathbb {C}))\setminus \sigma _{b-1}(X_{m,d}(\mathbb {C})))$ (\cite{l}, Chapter 7, \cite{o}, and references therein).

In section \ref{S2} we study the bivariate case and prove that $d$ is a typical rank for every border rank, i.e. we prove the following result.

\begin{theorem}\label{u1}
Fix integers $b, a, d$ such that $2 \le b \le (d+2)/2$, $d\ge $ and $1 \le a \le b/2$. Fix $Q_1,\dots ,Q_a\in \mathbb {P}^1(\mathbb {C})$
and $P_1,\dots ,P_{b-2a}\in \mathbb {P}^1(\mathbb {R})$ such that $P_i\ne P_j$ for all $i\ne j$
and $\sharp (\{Q_1,\dots ,Q_a,\sigma (Q_1),\dots ,\sigma (Q_a)\}) =2a$. Set
$$A:= \{Q_1,\dots ,Q_a,\sigma (Q_1),\dots ,\sigma (Q_a),P_1,\dots ,P_{b-2a}\}$$and $M:= \langle \nu _d(A)\rangle (\mathbb {R})$.
$M$ is a $(b-1)$-dimensional real vector space and there is a non-empty open subset of $M$ in the euclidean topology such that
$r_{\mathbb {R}}(P) =d$ for all $P\in U$.
\end{theorem}

We recall that in the bivariate case $d$ is the maximum of all real ranks. Hence knowing that $d$ is a typical rank for each border rank $\ne 1$ is the worst news we
could get on this subject. We prove a finer result which shows that the degree $d$ bivariate polynomials with real symmetric rank $d$ are ubiquitous (Theorem \ref{u2} and Remark \ref{u3}).

In section \ref{S3} we study the multivariate case. We first consider the border ranks $\le 7$. In each case we give all the typical ranks for
border rank $b\le 7$ if, say,  $m \ge \max \{2,b-1\}$ and $d$ is large (see Theorems \ref{w2},\ldots ,\ref{w7}). We also describe all the real ranks when $b=2$ (see Theorem \ref{w2}). For every $m\ge 2$, $b\ge 2$ and
for all $d\ge 2b-1$ we prove that $b$ and $b+d-2$ are the two smallest typical ranks for the border rank $b$ (Theorem \ref{w8}).

Thanks are due to a referee for useful suggestions.

\section{The multivariate case}\label{S3}
Let $\sigma
: \mathbb {C} \to \mathbb {C}$ be the complex conjugate and write $\sigma$ also for the involution induced by conjugation on each projective space
$\mathbb {P}^y(\mathbb {C})$. Hence $\mathbb {P}^y(\mathbb {R}) =\{P\in \mathbb {P}^y(\mathbb {C}):\sigma (P)=P\}$.

\begin{remark}\label{w0}
Fix integers $m\ge 1$, $b >0$, $d\ge 3$ such that $\sigma _b(X_{m,d}(\mathbb {C}) )\subsetneq \mathbb {P}^r(\mathbb {C})$ (e.g. assume $b(m+1) < {m+d \choose m}$).
Obviously $b$ is a typical rank of $\sigma _b(X_{m,d}(\mathbb {C}))(\mathbb {R})$.
\end{remark}

\begin{remark}\label{w1}
Fix positive integers $m, d, b$ such that $b\ge 2$ and $d\ge 2b-1$. We have $\dim _{\mathbb {C}}(\sigma _b(X_{m,d}(\mathbb {C})))= b(m+1)-1$. For any $P\in (\sigma _b(X_{m,d}(\mathbb {C}))\setminus \sigma _{b-1}(X_{m,d}(\mathbb {C})))$
there is a unique zero-dimensional scheme $Z\subset \mathbb {P}^m(\mathbb {C})$ such that $\deg (Z)=b$ and $P\in \langle \nu _d(Z)\rangle$ (\cite{bgi}, Proposition 11, \cite{bgl}, Lemma 2.1.6 and proof of Theorem 1.5.1, \cite{bb1}, Remark 1 and Lemma 1). Now
assume $P\in \sigma _b(X_{m,d}(\mathbb {C}))(\mathbb {R})$. Since $\sigma (P)=P$, the uniqueness of $Z$ implies $\sigma (Z)=Z$, i.e. the scheme $Z$ is defined over $\mathbb {R}$.
Hence the finite set $J:= Z_{red}$ is defined over $\mathbb {R}$. Set $e:= \sharp (J)$. Since $J$ is defined over $\mathbb {R}$, there are an integer
$a$ such that $0\le 2a \le e$, distinct points $Q_1,\dots ,Q_a\in  \mathbb {P}^m(\mathbb {C})\setminus \mathbb {P}^m(\mathbb {R})$ and $e-2a$ distinct points
$P_j\in  \mathbb {P}^m(\mathbb {R})$, $1\le j \le e-2a$, such that $J = \{P_1,\dots ,P_{e-2a},Q_1,\sigma (Q_1),\dots
,Q_a,\sigma (Q_a)\}$; the only restriction is that $Q_i\ne \sigma (Q_j)$ for all $i, j$. Now we vary $J\subset \mathbb
{P}^m(\mathbb {C})$ fixing $e$ and $a$, i.e. we vary $P_1,\dots ,P_{e-2a}\in \mathbb {P}^m(\mathbb {R})$ and $Q_1,\dots
,Q_a\in  \mathbb {P}^m(\mathbb {C})\setminus \mathbb {P}^m(\mathbb {R})$ with the restrictions $P_i\ne P_j$ for all $i\ne j$,
$Q_i\ne Q_j$ for all $i\ne j$ and $Q_h\ne \sigma (Q_k)$ for all $h,k$. Notice that $\dim (\langle \nu _d(J)\rangle )= e-1$ for
any such $J$, because $d\ge e-1$. If $d \ge 2e-1$ and $J\ne J'$, then we also get $\langle \nu _d(J)\rangle \cap \langle \nu
_d(J')\rangle = \langle \nu _d(J\cap J')\rangle$. Hence in this way we get an $(me+e-1)$-dimensional real manifold of $\sigma
_e(X_{m,d}(\mathbb {C}))(\mathbb {R})$. We will say that $(e-2a,a)$ is the type of $J$ or of this $(me+e-1)$-dimensional real
manifold or of any $P\in \langle \nu _d(J)\rangle (\mathbb {R})\setminus (\cup _{J'\subsetneq J} \langle \nu (J')\rangle )$.
\end{remark}

\begin{lemma}\label{v0}
Fix positive integers $m, t$. Let $S\subset \mathbb {P}^m(\mathbb {C})$ be a finite subset such that $\langle S\rangle =\mathbb {P}^m(\mathbb {C})$
and $h^1(\mathcal {I}_S(t)) >0$. Then $\sharp (S) \ge t+m+1$.
\end{lemma}

\begin{proof}
The lemma is true if $m=1$. Hence we may assume $m\ge 2$ and use induction on $m$. The case $t=0$ is true for arbitrary $m$, because $S$ spans $\mathbb {P}^m(\mathbb {C})$. Since $S$ spans $\mathbb {P}^m(\mathbb {C})$, there
is a hyperplane $H\subset \mathbb {P}^m(\mathbb {C})$ spanned by $m$ points of $S$. First assume $h^1(H,\mathcal {I}_{H\cap S}(t)) >0$. Since
$S\cap H$ spans $H$, the inductive
assumption gives $\sharp (H\cap S) \ge t+m$. Since $S$ spans $\mathbb {P}^m(\mathbb {C})$, we have $\sharp (S)>\sharp (S\cap H)\ge m+t$. Now assume $h^1(H,\mathcal {I}_{H\cap S}(t))=0$. The Castelnuovo's exact sequence
$$0 \to \mathcal {I}_{S\setminus S\cap H}(t-1) \to \mathcal {I}_S(t) \to \mathcal {I}_{S\cap H,H}(t) \to 0$$
gives $h^1(\mathcal {I}_{S\setminus S\cap H}(t-1))>0$. Hence $\sharp (S\setminus S\cap H) \ge t+1$ (obvious if $t=1$, \cite{bgi}, Lemma 34, if $t-1>0$). Since $\sharp (S\cap H) \ge m$, we get $\sharp (S) \ge t+m+1$.
\end{proof}

We recall the following weak form of \cite{b0}, Theorem 1.

\begin{lemma}\label{vv0}
Fix positive integers $m, t$. Let $S\subset \mathbb {P}^m(\mathbb {C})$, $m\ge 4$, be a finite subset such that $h^1(\mathcal {I}_S(t)) >0$. Assume $\sharp (S) \le 4t+m-5$. Then either there is a line $T\subset \mathbb {P}^m(\mathbb {C})$ such that $\sharp (S\cap T) \ge t+2$
or there is a plane $E\subset \mathbb {P}^m(\mathbb {C})$ such that $\sharp (E\cap S) \ge 2t+2$ or there is a $3$-dimensional linear subspace $F\subset \mathbb {P}^m(\mathbb {C})$ such that $\sharp (F\cap S) \ge 3t+2$.
\end{lemma}

\begin{lemma}\label{v1}
Fix integers $k \ge 1$, $e\ge 0$, $t \ge 1$, $m$ such that
$m \ge 2k-1+e$. Fix $k$ lines $L_1,\dots ,L_k\subset \mathbb {P}^{m}(\mathbb {C})$, a set $E\subset L_1\cup \cdots \cup L_k$ such that $\sharp (E\cap L_i) \le t+1$ for all $i$
and a set $F\subset \mathbb {P}^{m}(\mathbb {C})$ such that  $\sharp (F) = e$ and  $\dim (\langle L_1\cup \cdots \cup L_k\cup F\rangle ) =2k-1+e$. Then $h^1(\mathcal {I}_{E\cup F}(t))=0$.
\end{lemma}

\begin{proof}
Set $N:= \langle L_1\cup \cdots \cup L_k\rangle$ and $\Lambda := \langle N\cup F\rangle$. Since $\dim (\Lambda )=2k-1+e$, we have
$\dim (N) =2k-1$ and $F\cap N =\emptyset$. Since
$E\cup F\subset \Lambda$ and $h^1(\Lambda ,\mathcal {I}_{E\cup F,\Lambda}(t)) = h^1(\mathcal {I}_{E\cup F}(t))$, it is sufficient to prove
the lemma
in the case $m = 2k-1+e$. Since the case $k=1$ and $e=0$ is obvious, we may assume $k+e \ge 2$ and use induction on the integer $k+e$. 

\quad (a) Assume $e>0$ and fix $P\in F$. Set $F':= F\setminus \{P\}$. Notice that $H:= \langle L_1\cup \cdots \cup L_k\cup F' \rangle$ is a hyperplane of $\Lambda$ and that $\{P\} = (E\cup F)\setminus (E\cup F)\cap H$.
Hence we have an exact sequence on $\Lambda$:
\begin{equation}\label{eqv1}
0 \to \mathcal {I}_{\{P\}}(t-1) \to \mathcal {I}_{E\cup F}(t) \to \mathcal {I}_{E\cup F',H}(t) \to 0
\end{equation}
Since $t-1\ge 0$, we have $h^1(\mathcal {I}_{\{P\}}(t-1))=0$. The inductive assumption gives $h^1(H,\mathcal {I}_{E\cup F',H}(t))=0$. Hence (\ref{eqv1}) gives
$h^1(\mathcal {I}_{E\cup F}(t))=0$.

\quad (b) Assume $e=0$ and hence $k\ge 2$ and $F=\emptyset$. Set $N':= \langle L_1\cup \cdots \cup L_{k-1}\rangle$. Since
$\dim (N) =2k-1$, we have $\dim (N')=2k-3$. Hence there is a hyperplane $M$ containing $N'$
 and a point of $E\cap L_k$ (or just containing $N'$ if $E\cap L_k=\emptyset$). Since $\sharp (E\setminus E\cap M) \le t$, we
 have $h^1(\mathcal {I}_{E\setminus E\cap M}(t-1)) =0$. Look at the following exact sequence on $N$:
\begin{equation}\label{eqv2}
0 \to \mathcal {I}_{E\setminus E\cap M}(t-1) \to \mathcal {I}_E(t) \to \mathcal {I}_{E\cap M,M}(t) \to 0
\end{equation}
Since $\sharp (E\setminus E\cap M)\le t$, we have $h^1(\mathcal {I}_{E\setminus E\cap M}(t-1))=0$ (e.g., by \cite{bgi}, Lemma 34). Hence (\ref{eqv2}) gives
$h^1(\mathcal {I}_E(t))=0$.
\end{proof}

\begin{lemma}\label{ww1}
Fix integers $a\in \{1,2,3\}$ and $e$ with $0 \le e \le 7-2a$. Assume $m\ge 2a+e-1$
and $d\ge 4a+2e$. Fix $Q_i\in   \mathbb {P}^m(\mathbb {C})$, $1\le i \le a$, such
that $\sharp (\{Q_1,\dots ,Q_a,\sigma (Q_1),\dots ,\sigma (Q_a)\})=2a$ and $\dim (\langle \{Q_1,\dots ,Q_a,\sigma (Q_1),\dots ,\sigma (Q_a)\}\rangle
=2a-1$. Set $D_i:= \langle \{Q_i,\sigma (Q_i)\rangle$. If $e=0$, then set $A:= \{Q_1,\dots ,Q_a,\sigma (Q_1),\dots ,\sigma (Q_a)\}$. If $e>0$, then take
$P_1, \cdots P_e\in \mathbb {P}^m(\mathbb {R})$
such that  $\sharp (A)=2a+e$, where
$$A:=  \{P_1,\dots ,P_e,Q_1,\dots ,Q_a,\sigma (Q_1),\dots ,\sigma (Q_a)\}.$$ Fix $P\in \langle \nu _d(A)\rangle (\mathbb {R})$ such that $P\notin \langle \nu _d(A')\rangle$ for any $A'\subsetneq A$. Then $A$ is the only scheme
evincing $r_{\mathbb {C}}(P)$, $r_{\mathbb {R}}(P)=2a+e$ and every $B\subset \mathbb {P}^m(\mathbb {R})$ evincing $r_{\mathbb {R}}(P)$ contains $d$ points on each $D_i$
and (if $e>0$), $P_1,\dots ,P_e$.
\end{lemma}

\begin{proof}
Notice that $\dim (\langle \nu _d(A)\rangle )=\sharp (A)-1$. Since $A$ is a finite set, it has finitely many proper subsets.
Hence $P$ exists and the set of all such points $P$ is an open and dense subset of the real vector space $\langle \nu
_d(A)\rangle (\mathbb {R})$. The set $A$ is unique (\cite{bgl}, Theorem 1.5.1, or Remark \ref{w1}). 

\quad (a) Fix $A'\subsetneq A$ such that $A'\ne \emptyset$. Since $P\in \langle \nu _d(A)\rangle$ and $P\notin \langle \nu _d(A_1)\rangle$ for any $A_1\subsetneq A$, there
is a unique $P'\in \langle \nu _d(A')\rangle$ such that $P\in \langle \{P'\}\cup \nu _d(A\setminus A')\rangle$. Since
$\sigma (A) =A$ and $\sigma (P)=P'$, if $A'$ is
defined over $\mathbb {R}$ the uniqueness of $P'$
implies $\sigma (P')=P'$.

\quad (b) Take $B\subset \mathbb {P}^m(\mathbb {R})$ evincing $r_{\mathbb {R}}(P)$. Hence $P\notin \langle \nu _d(B')\rangle$ for
any $B'\subsetneq B$. Since $B\ne A$, \cite{bb1}, Lemma 1, gives $h^1(\mathcal {I}_{A\cup B}(d)) >0$. Set $W_0:= A\cup B$. Set $D_i := \langle \{Q_i,\sigma (Q_i)\}\rangle$.  Each $D_i$ is a line defined over $\mathbb {R}$. We have $\langle \cup _{i=1}^{a} D_i\rangle
= \langle \{Q_1,\sigma (Q_1),\dots ,Q_a,\sigma (Q_a)\}\rangle$ and hence 
$$\dim (\langle \cup _{i=1}^{a} D_i\rangle )
= \dim (\langle \{Q_1,\sigma (Q_1),\dots ,Q_a,\sigma (Q_a)\rangle )=2a-1.$$ By step (a) there is a unique $O_i\in \langle \nu _d(D_i)\rangle (\mathbb {R})$
such that $P\in \langle \{O_1,\dots ,O_a\}\cup \nu _d(\{P_1,\dots ,P_e\})\rangle$. Since each $O_i$ has rank $\le d$ with respect to the rational normal
curve $\nu _d(D_i)$ (\cite{co}, Proposition 2.1), we get $\sharp (B) \le e +ad$. Hence $\sharp (A\cup B) \le 2e +ad +2a$. If $a\le 2$, then we get $\sharp (A\cup B) \le 2d+10$.
If $a=3$, then $\sharp (A\cup B) \le 3d+7$. If $(a,e)=(3,1)$ (resp. $(a,e) =(3,0)$) we have $d\ge 14$ (resp. $d\ge 12$) and $m\ge 6$ (resp. $m\ge 5$). Hence if $a=3$ we have
$2m-4 +3d > \sharp (A\cup B)$, unless $a=3$, $e=0$ and $m=5$. 

\quad (c) Let $H_1\subset \mathbb {P}^m(\mathbb {C})$ be a hyperplane such that $b_1:= \sharp (W_0\cap H_1)$ is
maximal. Set $W_1:= W_0\setminus W_0\cap H_1$. For each integer $i\ge 2$ define recursively the hyperplane $H_i\subset \mathbb {P}^m(\mathbb {C})$, the integer $b_i$
and the set $W_i$ in the following way. Let $H_i\subset \mathbb {P}^m(\mathbb {C})$ be a hyperplane such that $b_i:= \sharp (W_{i-1}\cap H_i)$ is
maximal. Set $W_i:= W_{i-1}\setminus W_{i-1}\cap H_i$. For each integer $i\ge 1$ we have an exact sequence
\begin{equation}\label{eqw1}
0 \to \mathcal {I}_{W_i}(d-i) \to \mathcal {I}_{W_{i-1}}(d+1-i) \to  \mathcal {I}_{W_{i-1}\cap H_i,H_i}(d+1-i) \to 0
\end{equation}
Since $h^1(\mathcal {I}_{W_0}(d)) >0$, (\ref{eqw1}) implies the existence of an integer $i>0$ such that $h^1(H_i,\mathcal {I}_{W_{i-1}\cap H_i,H_i}(d+1-i)) >0$.
We call $g$ the minimal such an integer. The sequence $\{b_i\}_{i\ge 0}$ is non-decreasing. Any $m$ points
of $\mathbb {P}^m(\mathbb {C})$ are contained in a hyperplane.
Hence if
$b_i\le m-1$, then
$b_{i+1}=0$. Since $\sharp (W_0) \le 2a+ad+e$ and $m \ge 2a+e-1$, we get $b_i=0$ for all $i>(d+2)/2$. Hence $g\le (d+2)/2$.
Since $h^1(H_g,\mathcal {I}_{W_{g-1}\cap H_g}(d+1-g)) >0$, we have $b_g\ge d+3-g$ (\cite{bgi}, Lemma 34). 
Since $b_i\ge b_g$ for all $i\le g$, we get $g(d+3-g) \le 2a+ad+e$. Until step (d) we assume $g\ge 2$. 
Assume for the moment $b_g\ge 2d+4-2g$. Since $b_i\ge b_g$ for all $i\le g$, we get $g(2d+4-2g) \le 2a+ad+e \le ad+7$. Set $\psi (t) := 2t(d+2-t)$.
The real function $\psi (t)$ is increasing if $t\le (d+2)/2$ and decreasing if $t> (d+2)/2$. Since $2 \le g \le (d+2)/2$, $\psi (2) =4d > ad+7$
and $\psi (d+2)/2 = (d+2)^2 > ad+7$, we get a contradiction. Hence $b_g \le 2(d+1-g)+1$. By \cite{bgi}, Lemma 34, there is a line $T\subset H_i$ such
that $\sharp (T\cap W_{g-1})\ge d+3-g$. Since $b_g >0$, $W_{g-2}$ spans $\mathbb {P}^m$. Hence $b_{g-1} \ge (m-2)+d+3-g$. Since $b_i \ge b_{g-1}$ for
all $i\le g-1$, we get $2a+ad+e \le (g-1)(m-2) +g(d+3-g)$. Set $\phi (t):= t(d+1+m -t) -m+2$. Since $\phi (t)$ is increasing if $t\le (d+1+m)/2$, $2 \le g \le (d+2)/2$,
and $\phi (3) = 2m-4+3d > \sharp (A\cup B)$ (unless $m=5$, $a=3$ and $e=0$), we get $g=2$, unless $a=3$, $e=0$ and $m=5$; in this case we only get $g \le 3$.

\quad (c1) Assume for the moment $m=5$, $a=3$, $e=0$ and $g=3$. We have $b_3 \ge d$ and $b_1\ge b_2\ge d+3$. Hence $b_4=0$, $b_3=d$, $b_2 = b_1 = d+3$ and
$W_2\subset T$ and $\sharp (W_2) =d$. Let $H'_1 \subset \mathbb {P}^5(\mathbb {C})$ be a hyperplane containing $T$ and with $b'_1:= \sharp (H'_1\cap W_0)$ maximal
among the hyperplanes containing $T$. Set $W'_1:= W_0\setminus W_0\cap H'_1$. If $h^1(H'_1,\mathcal {I}_{W_0\cap H'_1,H'_1}(d)) >0$, then go to step (d). Now assume $h^1(H'_1,\mathcal {I}_{W_0\cap H'_1,H'_1}(d))=0$.
From the exact sequence (\ref{eqw1}) with $i=1$ and $H'_1$ instead of $H_1$ we get
$h^1(\mathcal {I}_{W'_1}(d-1)) >0$. Since $W_0$ spans $\mathbb {P}^5$, we have $b'_1 \ge d+3$. Since
$b_1\ge b'_1$, we get $b'_1:= b_1$. Hence any $J\subset W_0\setminus W_0\cap T$ with $\sharp (J) \le 4$ is linearly independent.
Since $\sharp (W'_1) = 2d+3$ and $m=5$, we easily get $h^1(\mathcal {I}_{W'_1}(d-1)) =0$ (Lemma \ref{v1} is stronger), a contradiction. 

\quad (c2) Now assume $g=2$. Set $U_0:= W_0$. Let $M_1$ be a hyperplane containing $T$ and with $c_1:= \sharp (M_1\cap W_0)$ maximal among the hyperplanes
containing $T$. Set $U_1:= U_0\setminus U_0\cap M_1$. Define recursively the hyperplane $M_i\subset \mathbb {P}^m(\mathbb {C})$, the integer $c_i$
and the set $U_i$ in the following way. Let $M_i\subset \mathbb {P}^m(\mathbb {C})$ be a hyperplane such that $c_i:= \sharp (U_{i-1}\cap M_i)$ is
maximal. Set $U_i:= U_{i-1}\setminus U_{i-1}\cap M_i$. For each integer $i\ge 1$ we have an exact sequence like (\ref{eqw1}) with $U_{i-1}$ and $U_i$ instead
of $W_{i-1}$ and $W_i$ and with $M_i$ instead of $H_i$. Hence there is a minimal integer $f\ge 1$ such that $h^1(M_f,\mathcal {I}_{U_{f-1}\cap M_f,M_f}(d+1-f)) >0$.
Assume for the moment $f\ge 2$. Notice that $c_1 \ge d+m-1$, that $c_{i+1}\le c_i$ for all $i\ge 2$ and that $c_m \ge m$ for all $m<f$. As in the case with $g$
we get a contradiction, unless $f=2$. First assume $c_2 \ge 2d$. Hence $c_1+c_2 \ge 4d > 3d+7 \ge \sharp (A\cup B)$, a contradiction. Now assume
$c_2\le 2d-1$. Since $h^1(M_2,\mathcal {I}_{U_1\cap M_2,M_2}(d-1)) >0$, there is a line $T'\subset M_2$ such that
$\sharp (T'\cap U_0) \ge d+1$. Hence $\sharp ((T'\cup T)\cap W_0) \ge 2d+2$. Since $b_2>0$, $W_0$ spans $\mathbb {P}^m$. Hence there is a hyperplane
containing $T'\cup T$ and $m-4$ further points of $W_0\setminus W_0\cap (T\cup T')$. Hence $b_1 \ge 2d+m-2$. Since $b_2\ge d+1$, we get
$ad+7 \le 3d+m-1$. Hence $a=3$ and $5 \le m \le 8$. Take a hyperplane $H''$ containing $T'\cup T$ and at least $m-4$ further points and assume $h^1(H'',\mathcal {I}_{W_0\cap H'',H''}(d)) =0$.
An exact sequence (\ref{eqw1}) with $H''$ instead of $H_1$ gives $h^1(\mathcal {I}_{W_0\setminus W_0\cap H''}(d-1)) >0$. The proof of
the inequality $c_2\le 2d-1$ gives $\sharp (W_0\setminus W_0\cap H'') \le 2d-1$. Hence there is a line $T''$ such that $\sharp (T''\cap (W_0\setminus W_0\cap H''))\ge d+1$.
If $m\ge 6$, then there is a hyperplane containing $T\cup T'\cup T''$ and hence $b_1\ge 3d+3$; hence $b_1+b_2 \ge 4d+4 > 3d+7$, a contradiction. Now assume
$m=5$, $a=3$ and $e=0$ and $\dim (\langle T\cup T'\cup T' \rangle )=5$. See step (d5) for this case.

\quad (d) In this step we cover the case $g=1$, the case $g \ge 2$ and $f=1$ the case $h^1(H'_1,\mathcal {I}_{W_0\cap H'_1,H'_1}(d)) >0$,
the case $h^1(H''_1,\mathcal {I}_{W_0\cap H''_1,H''_1}(d)) >0$ and conclude the proof of the case $g=2$, $m=5$, $a=3$ and $e=0$. In these cases (except the last one)
we have a hyperplane $H\subset \mathbb {P
}^m(\mathbb {C})$ defined over $\mathbb {R}$ and with $h^1(H,\mathcal {I}_{H\cap W_0}(d)) >0$.

\quad (d1) Assume for the moment $W_0\subset H$. In this case we may use induction on $m$, because the inclusion $A\subset H$ implies $m > 2a+ad+e-1$. Hence
from now on we assume that $W_0 \setminus W_0\cap H \ne \emptyset$. We also reduce (decreasing if necessary $m$) to the case in which $W_0$ spans $\mathbb {P}^m(\mathbb {C})$.
Since $W_0$ spans $\mathbb {P}^m(\mathbb {C})$, we may take $H$ with the additional condition that $W_0\cap H$ spans $H$.

\quad (d2) Assume for the moment $h^1(\mathcal {I}_{W_0\setminus W_0\cap H}(d-1))=0$.
By \cite{bb3}, Lemma 5, we have $A\setminus A\cap H = B\setminus B\cap H$. Hence $\sharp (B\setminus B\cap H) \le e$ and $H$ contains
$\langle \{Q_1,\sigma (Q_1),\dots ,Q_a,\sigma (Q_a)\}\rangle$. Since $A\subset H$ and $W_0\nsubseteq H$, we have
$e>0$. Set $e' := \sharp (W_0\setminus W_0\cap H)$ and write $F_1:= W_0\setminus W_0\cap H$, $A_1:= A\setminus F_1$ and $B_1:= B\setminus F_1$.
Take either $\mathbb {K} = \mathbb {C}$ or $\mathbb {K} = \mathbb {R}$.
Since $d \ge 2a+e-1$, we have $\langle \nu _d(A_1)\rangle (\mathbb {K})\cap \langle \nu _d(F_1)\rangle (\mathbb {K}) = \emptyset$.
Since $d\ge ad+e-1$, we have $\langle \nu _d(B_1)\rangle (\mathbb {K})\cap \langle \nu _d(F_1)\rangle (\mathbb {K}) = \emptyset$.
Since $d\ge \sharp (A\cup B)-1$, we
have $\langle \nu _d(A)\rangle  (\mathbb {K})\cap \langle \nu _d(B)\rangle (\mathbb {K}) =\langle \nu _d(A\cap B)\rangle  (\mathbb {K})$.
Hence taking $A_1$ and $B_1$ instead of $A$ and $B$ we reduce to the case $(a',e') =(a,e-\sharp (F_1))$ and in this case we get that $W_0\setminus F_1$
is contained in a hyperplane (a case inductively solved). Hence from now on we assume $h^1(\mathcal {I}_{W_0\setminus W_0\cap H}(d-1))>0$. By \cite{bgi}, Lemma 34, we
have $\sharp (W_0) -\sharp (W_0\cap H) \ge d+1$. Lemma \ref{v0} gives $\sharp (W_0\cap H) \ge d+m-2$.

\quad (d3) Assume for the moment $\sharp (W_0) -\sharp (W_0\cap H) \ge 2d$. Hence
$2a+e+ad \le 3d+m-2$. Hence $a=3$ and $5\le m\le 6$. See Step (d7).

\quad (d4) Now assume $\sharp (W_0\setminus W_0\cap H) \le 2d-1$. By \cite{bgi}, Lemma 34, there is a line $T_1\subset \mathbb {P}^m(\mathbb {C})$ such
that $\sharp (T_1\cap (W_0\setminus W_0\cap H)) \ge d+1$. Let $R_1\subset \mathbb {P}^m(\mathbb {C})$ be a hyperplane such that $R_1\supset T_1$ and
$m_1:= \sharp (W_0\cap R_1)$ is maximal among the hyperplanes containing $T_1$. Since  $\sharp (T_1\cap (W_0\setminus W_0\cap H)) \ge d+1$, we
have $m_1\ge d+m-1$. Assume for the moment $h^1(\mathcal {I}_{W_0\setminus W_0\cap R_1}(d-1))=0$. By \cite{bb3}, Lemma 5, we have $A\setminus A\cap H = B\setminus B\cap H$. Hence $\sharp (B\setminus B\cap R_1) \le e$ and $R_1$ contains
$\langle \{Q_1,\sigma (Q_1),\dots ,Q_a,\sigma (Q_a)\}\rangle$. Since $A\subset R_1$ and $W_0\nsubseteq R_1$, we have
$e>0$. As in step (d2) we reduce to a case with a smaller $e$, say $e'$, for which we run all the proof from step (b) on. In case (d1) we see that we get a smaller $m$ for $e'<e$. Hence after finitely many steps we run only with cases with  $h^1(\mathcal {I}_{W_0\setminus W_0\cap R_1}(d-1)) >0$.

Now assume $h^1(\mathcal {I}_{W_0\setminus W_0\cap R_1}(d-1)) >0$. See step (d7) for the case $\sharp (W_0\setminus W_0\cap R_1))\ge 2d$.
Here we assume $\sharp (W_0\setminus W_0\cap R_1)) \le 2d-1$. Hence there is a line $T_2$ such that $\sharp (T_2\cap (W_0\setminus W_0\cap R_1)) \ge d+1$
(\cite{bgi}, Lemma 34). Let $R_1\subset \mathbb {P}^m(\mathbb {C})$ be a hyperplane such that $R_1\supset T_1\cup T_2$ and $W_0\cap R_2$ spans $R_2$. Notice that $\sharp (R_2\cap W) \ge m+2d-2$. First
assume $h^1(\mathcal {I}_{W_0\setminus W_0\cap R_2}(d-1))=0$. We conclude as at the beginning of step (d2). Now assume $h^1(\mathcal {I}_{W_0\setminus W_0\cap R_2}(d-1)) >0$.
Since $4d-m-1 > \sharp (W_0)$, we have $\sharp (W_0\setminus W_0\cap R_2) \le 2d-1$. Hence $\sharp (W_0\setminus W_0\cap R_2)) \ge d+1$ (it implies $a=3$)
and there is a line $T_3$ such that $\sharp (T_3\cap (W_0\setminus W_0\cap R_2) )\ge d+1$. First assume that either $m\ge 6$ or $\dim (\langle T_1\cup T_2\cup T_3\rangle )\le 4$.
In these cases there is a hyperplane $R_3$ containing $T_1\cup T_2\cup T_3$ and spanned
by some of the points of $W_0$. Since $T_1\cup T_2\subset R_2$ and  $\sharp (T_3\cap (W_0\setminus W_0\cap R_2) )\ge d+1$, we have $\sharp (W_0\cap R_3) \ge 3d+3 +(m-5)$.
Hence $\sharp (W_0)\ge 3d+m-2$. Hence either $m=5$, $a=3$ and $e=0$ or $m=6$, $a=3$ and $e=1$.

\quad (d5) Now assume $m=5$, $a=3$, $e=0$ and $\dim (\langle T_1\cup T_2\cup T_3\rangle )=5$. Hence $T_i\ne T_j$ for all $i\ne j$. Since $\sharp (A\cap T_j) \le 2$,
we have $\sharp (B\cap T_j) \ge 2$. Hence each line $T_j$ is defined over $\mathbb {R}$. Hence either $T_j\cap A =\emptyset$ or $T_j =\langle \{Q_h,\sigma (Q_h)\}\rangle$
for some $h\in \{1,2,3\}$. All triples of complex lines in $\mathbb {P}^5(\mathbb {C})$ whose union spans $\mathbb {P}^5(\mathbb {C})$ are projectively equivalent.
Hence we immediately see that
the sheaf $\mathcal {I}_{T_1\cup T_2\cup T_3}(2)$ is spanned by its global sections.
Since $W_0$ is finite, we get the existence of a complex quadric hypersurface $J$ such that $T_1\cup T_2\cup T_3\subset J$ and $J\cap W_0=
(T_1\cup T_2\cup T_3)\cap W_0$. Since $\deg (T_i\cap W_0)\ge d+1$ for all $i$ and $T_i\cap T_j =\emptyset$ for all $i\ne j$, we have
$\sharp (W_0\setminus W_0\cap J) \le 3$. Hence $h^1(\mathcal {I}_{W_0\setminus W_0\cap J}(d-2)) =0$. As in \cite{bb3}, Lemma 5, we get
$A\setminus A\cap J = B\setminus B\cap J$. Since $A\cap B=\emptyset$, we get $A\cup B \subset J$. Since
$J\cap W_0 = (T_1\cup T_2\cup T_3)\cap W_0$. Hence $A\cup B\subset T_1\cup T_2\cup T_3$. Since $A\subset T_1\cup T_2\cup T_3$ we get that
(up to renaming the points $Q_1,Q_2,Q_3$) $T_i = \langle \{Q_i,\sigma (Q_i)\}\rangle$. To conclude the lemma in this case
we only need to prove that $\sharp (B\cap T_i) =d$ for all $i$. Up to now we only know that $\sharp (B\cap T_i) \ge d-1$. Since $\sharp (B) \le 3d$, it is
sufficient to prove that $\sharp (B\cap T_i) \ge d$ for all $i$. Assume that this is not the case and that, up to a permutation of the indices $1,2,3$, there
is $h\in \{1,2,3\}$ such that $\sharp (T_i\cap B) =d-1$ if and only if $i\ge h$. Since $A\subset T_1\cup T_2\cup T_3$,
there are $O_i\in \langle \nu _d(T_i)(\mathbb {C})\rangle$ such that $P\in \langle \{O_1,O_2,O_3\}\rangle$.
Since $P\notin \langle \nu _d(A')\rangle$ for any $A' \subsetneq A$, we get $P\notin \langle E\rangle$ for any $E\subsetneq A_i$ and that the points $O_i$ are
unique. The uniqueness of $O_i$ implies $O_i\in  \langle \nu _d(T_i)(\mathbb {R})\rangle$. Fix $i<h$ (if any). Take the union of $B\cap T_j$ for all $j\ne i$
and a set computing the real rank of $O_i$ with respect to the rational normal curve $\nu _d(T_i)$. Since $\sharp (B) = r_{\mathbb {R}}(P)$, we get
$\sharp (B\cap T_i) = d$ for all $i<h$ (if any). There is a hyperplane $M\subset \mathbb {P}^m(\mathbb {C})$ such that $M$ contains $T_i$ for
all $i<h$ and exactly one point of $T_j\cap B$ for all $j\ge 4$. Hence $W_0\setminus W_0\cap M$ is the union of $d$ points of $T_j$ for all $j\ge h$. Lemma \ref{v1}
implies $h^1(\mathcal {I}_{W_0\setminus W_0\cap M}(d-1)) =0$. Hence $A\setminus A\cap M = B\setminus B\cap M$ (\cite{bb3}, Lemma 5). Hence $B\subset M$, a contradiction.

\quad (d6) Now assume $m= 6$, $a=3$ and $e=1$. This case is done as in step (d5), because $d-2 \ge 11$
and $\sharp (W_0\setminus W_0\cap (T_1\cup T_2\cup T_3)) \le 4$. We stated Lemma \ref{v1} in the case $e>0$ to allow its quotation here.

\quad (d7) In this step we conclude the proof of steps (d3) and (d4). We assume the existence of a hyperplane
$H\subset \mathbb {P}^m(\mathbb {C})$ such that $\sharp (W_0) -\sharp (W_0\cap H) \ge 2d$, $h^1(\mathcal {I}_{W_0\setminus W_0\cap H}(d-1)) >0$
and $W_0\cap H$ spans $H$. Hence $a=3$ and it is sufficient to check the cases $m=5$, $e=0$ and $m=6$, $e=1$. By Lemma \ref{vv0}
there are $i\in \{1,2,3\}$ and an $i$-dimensional linear subspace $J_i\subset \mathbb {P}^m(\mathbb {C})$ such that $\sharp (T_1\cap (W_0\setminus W_0\cap H)) \ge i(d-1)+2$.
If $i=1$, then we continue as in step (d4), just using the line $T_1$. In all other cases take a hyperplane $R$
containing $J_i$ and with maximal $\alpha :=\sharp (R\cap W_0)$. We have $\alpha \ge i(d-1)+2+m-i$. In step (d2) we proved
that a contradiction arises unless $h^1(\mathcal {I}_{W_0\setminus W_0\cap R}(d-1)) >0$. Hence $\sharp (W_0)-\alpha \ge d+1$. If $i=3$, then we get
$\sharp (W_0) \ge d+1 + 3d-1+m-4 $, a contradiction. Now assume $i=2$.  Since $2d+\alpha > \sharp (W_0)$, we have $\sharp (W_0\setminus W_0)\le 2d-1$.
Since $h^1(\mathcal {I}_{W_0\setminus W_0\cap R}) >0$, there is a line $T_4$ such that $\sharp (T\cap (W_0\setminus W_0\cap R)\ge d+1$.
Since $m \ge 5$ and $R\supset J_2$, there is a hyperplane containing $J_2$ and $T_4$. The maximality property of $\alpha$ gives $\alpha \ge 3d+1$.
Hence $\sharp (W_0)\le (3d+1)+(d+1)$, a contradiction.\end{proof}

\begin{theorem}\label{w2}
Fix integers $m\ge 1$ and $d\ge 3$ and any $P\in \sigma _2(X_{m,d}(\mathbb {C}))(\mathbb {R})\setminus X_{m,d}(\mathbb {R})$. Then either $r_{\mathbb {R}}(P)=2$ or
$r_{\mathbb {R}}(P)=d$ and both $2$ and $d$ are typical real rank for $\sigma _2(X_{m,d})$. We have $r_{\mathbb {R}}(P)=d$ if and only
if either $P$ is a point of the tangential variety $\tau (X_{m,d}(\mathbb {C}))(\mathbb {R})$ or $r_{\mathbb {C}}(P)=2$
and the only $A\subset \mathbb {P}^m(\mathbb {C})$ evincing $r_{\mathbb {C}}(P)$ is of the form $\{Q,\sigma (Q)\}$ for some $Q\in \mathbb {P}^m(\mathbb {C})\setminus \mathbb {P}^m(\mathbb {R})$. $2$ and $d$ are the typical ranks of $\sigma _2(X_{m,d}(\mathbb {C}))(\mathbb {R})$.
\end{theorem}

\begin{proof}
By Remark \ref{w1} there is a unique scheme $Z\subset \mathbb {P}^m(\mathbb {C})$ such that $\deg (Z)=2$, $P\in \langle \nu _d(Z)\rangle$ and $Z$ is defined over $\mathbb {R}$. First assume that $Z$ is not reduced and set $\{O\}:= Z_{red}$. Since $\sigma (Z)=Z$, we
have $O\in \mathbb {P}^m(\mathbb {R})$ and $D:= \langle Z\rangle \subseteq \mathbb {P}^m(\mathbb {C})$ is a line
defined over $\mathbb {R}$. In this case we have $r_{\mathbb {C}}(P)=d$ (\cite{bgi}, Theorem 32). Hence $r_{\mathbb {R}}(P)\ge d$. Since $P\in \langle \nu _d(D)\rangle$, $r_{\mathbb {R}}(P)$
is at most the real rank of $P$ with respect to the rational normal curve $\nu _d(D)$, we have $r_{\mathbb {R}}(P)\le d$. Hence $r_{\mathbb {R}}(P)= d$. Now
assume that $Z$ is reduced. If $Z =\{P_1,P_2\} \subset \mathbb {P}^m(\mathbb {R})$ with $P_1\ne P_2$, then $r_{\mathbb {R}}(P)= 2$. Now
assume $Z = \{Q,\sigma (Q)\}$ for some $Q\in  \mathbb {P}^m(\mathbb {C})\setminus \mathbb {P}^m(\mathbb {R})$. Set $T:= \langle \{Q,\sigma (Q)\}$. Since
the line $T$ is defined over $\mathbb {R}$ and $P\in \langle \nu _d(T)\rangle$, the bivariate case gives $r_{\mathbb {R}}(P) \le d$ (\cite{co}, Proposition 2.1). 

\quad {\emph {Claim :}} Let $B\subset \mathbb {P}^m(\mathbb {R})$ be any set evincing $r_{\mathbb {R}}(P)$. Then $\sharp (B) = d$ and $B\subset T(\mathbb {R})$.

\quad {\emph {Proof of the Claim:}} Since $r_{\mathbb {R}}(P)\le d$, we have $\sharp (B)\le d$.
Set $A:= \{Q,\sigma (Q)\}$. Since $A\ne B$, $P\in \langle \nu _d(A)\rangle \cap \langle \nu _d(B)\rangle$, $P\notin \langle \nu _d(A')\rangle$
for any $A'\subsetneq A$ and $P\notin \langle \nu _d(B')\rangle$
for any $B'\subsetneq B$, we have $h^1(\mathcal {I}_{A\cup B}(d)) >0$ (\cite{bb1}, Lemma 1). Since $\sharp (A\cup B)\le d+2$, \cite{bgi}, Lemma 34, gives
$\sharp (A\cup B)=d+2$ and the existence of a line $T'\subset \mathbb {P}^m(\mathbb {C})$ such that $A\cup B\subset T'$. Since $\sharp (A\cup B)=d+2$, we have
$\sharp (B) =d$. Since $A$ spans $T$, we have $T'=T$.

The last part of Remark \ref{w1} gives that $2$ and $d$ are the typical ranks of $\sigma _2(X_{m,d}(\mathbb {C}))(\mathbb {R})$.\end{proof}

\begin{theorem}\label{w3}
Fix integers $m \ge 2$ and $d\ge 6$. Then $3$ and $d+1$ are the typical ranks of $\sigma _3(X_{m,d}(\mathbb {C}))(\mathbb {R})$.
\end{theorem}

\begin{theorem}\label{w4}
Fix integers $m\ge 3$ and $d\ge 8$. For all $m\ge 3$ the typical ranks of $\sigma _4(X_{m,d}(\mathbb {C}))(\mathbb {R})$ are $4$, $d+2$ and $2d$.
\end{theorem}

\begin{theorem}\label{w5}
Fix integers $m\ge 4$ and $d \ge 10$. The typical ranks of 
$\sigma _5(X_{m,d}(\mathbb {C}))(\mathbb {R})$ are $5$, $d+3$ and $2d+1$.
\end{theorem}

\begin{theorem}\label{w6}
Fix integers $m \ge 5$ and $d\ge 12$. The typical ranks of 
$\sigma _6(X_{m,d}(\mathbb {C}))(\mathbb {R})$ are $6$, $d+4$, $2d+2$ and $3d$.
\end{theorem}

\begin{theorem}\label{w7}
Fix integer $m \ge 6$ and $d\ge 14$. The typical ranks of 
$\sigma _7(X_{m,d}(\mathbb {C}))(\mathbb {R})$ are $7$, $d+5$, $2d+3$ and $3d+1$.
\end{theorem}

\vspace{0.3cm}

\qquad {\emph {Proofs of Theorems \ref{w3}, \ref{w4}, \ref{w5}, \ref{w6} and \ref{w7}.}} Fix
the border rank $b\in \{3,4,5,6,7\}$. Obviously $b$ is a typical rank for the border rank. Notice that in the statement of the theorem
concerning the border rank $b$ we assumed $m\ge b-1$. For any integer $a$ such that $1 \le a \le b/2$ we have $d \ge 4a +2(b-2a)$. Hence
the assumptions of Lemma \ref{ww1} are satisfied for the data $m,d,a,e:= b-2a$. Lemma \ref{ww1} says
that all typical ranks for the border rank $b$ are obtained taking an integer $a\in \{1,\dots ,\lfloor b/2\rfloor\}$, setting $e:= b-2a$ and then
describing a subset of $A\in \mathbb {P}^m(\mathbb {C})^b$ with $\sigma (A)=A$ and associated to the typical rank $ad +e$.\qed

\begin{theorem}\label{w8}
Fix integers $m\ge 2$, $b\ge 2$ and $d\ge 2b-1$. Then $b$ and $b+d-2$ are the first two typical ranks of $\sigma _b(X_{m,d}(\mathbb {C}))(\mathbb {R})$.
\end{theorem}

\begin{proof}
Obviously $b$ is the minimal typical rank. Take an integer $g$ such that $b<g \le b+d-2$ and take a sufficiently general
$P$ in an $(mb+b-1)$-dimensional open subset of $\sigma _b(X_{m,d}(\mathbb {C}))(\mathbb {R})\setminus \sigma _{b-1}(X_{m,d}(\mathbb {C}))(\mathbb {R})$ with real symmetric rank $g$. Since
$d\ge 2b-1$, there is a unique zero-dimensional scheme $A\subset \mathbb {P}^m(\mathbb {C})$ such that $\deg (A) =b$ and $P\in \langle \nu _d(A)\rangle$ (Remark \ref{w1}).
Moreover $P\notin \langle \nu _d(A')\rangle$ for any $A'\subsetneq A$. For a general $P$ we may assume that $A$ is reduced and that its Hilbert function
is the Hilbert function of a general set of $b$ points of $\mathbb {P}^m(\mathbb {C})$, i.e. $h^1(\mathcal {I}_A(t)) = \max \{0,b-{m+t \choose m}\}$ for all $t\in \mathbb {N}$.
In particular no $3$ of the points of $A$ are collinear.
The uniqueness of $A$ implies $\sigma (A)=A$. 
Take $B\subset \mathbb {P}^m(\mathbb {R})$ evincing $r_{\mathbb {R}}(P)$. We have $h^1(\mathcal {I}_{A\cup B}(d)) >0$ (\cite{bb1}, Lemma 1). Since $\sharp (A\cup B) \le 2b+d-2 \le 2d+1$,
there is a line $T\subset \mathbb {P}^m(\mathbb {C})$ such that $\sharp (A\cup B) \ge d+2$. Since $\sharp (A\cap T) \le 2$, we have $\sharp (B\cap T)\ge d$.
Since $\sharp (B\cap T)\ge 2$, the line $T$ is defined over $\mathbb {R}$. Since $A\cup B$ is a finite set, there is a hyperplane $H\subset \mathbb {P}^m(\mathbb {C})$
such that $H\supseteq T$ and  $H\cap (A\cup B\setminus (A\cup B)\cap T) =\emptyset$. Since $\sharp (A\cup B)-\sharp (H\cap (A\cup B)) \le d-1 \le d$,
we have $h^1(\mathcal {I}_{A\cup B\setminus H\cap (A\cup B)}(d-1))=0$. Hence $A\setminus A\cap H = B\setminus B\cap H$ (\cite{bb3}, Lemma 5), i.e.
$A\setminus A\cap T = B\setminus B\cap T$. Since $\sharp (A\cap T)\le 2$, $\sharp (A\cup B)\cap T)\ge d+2$ and $g\le b+d-2$,
we get $g = b+d-2$, $\sharp (A\cap T)=2$, $\sharp (B\cap T)=d$ and $A\cap B\cap T=\emptyset$. Hence $g =d+b-2$ and at least $b-2$ of the points of $A$ are real. Since $g>b$, not all points of $b$ are real. We get
that $A = \{Q,\sigma (Q)\}\cup (B\setminus B\cap T)$ for some $Q\in \mathbb {P}^m(\mathbb {C})\setminus \mathbb {P}^m(\mathbb {R})$. Conversely
take any $P'\in \langle \nu _d(A)\rangle (\mathbb {R})\setminus (\cup _{A'\subsetneq A}\langle \nu _d(A')\rangle )$, with $A = \{Q,\sigma (Q)\}\cup F$ for some
$F\subset \mathbb {P}^m(\mathbb {R})$ with $\sharp (F)=b-2$. Since $d \ge 2b-1$, we have $r_{\mathbb {C}}(P') =b$ and $P'$ has border rank $b$. Since
$d\ge 2b-1$, $A$ is the only set with cardinality $\le b$ such that $P'\in \langle \nu _d(A)\rangle (\mathbb {C})$. Since $Q\notin \mathbb {P}^m(\mathbb {R})$, we
have $r_{\mathbb {R}}(P') >b$. Varying $Q$ and $F$
we cover a non-empty open subset of $\sigma _b(X_{m,d}(\mathbb {C}))(\mathbb {R})$. The first part of the proof gives $r_{\mathbb {R}}(P') \ge d+b-2$.
Set $D:= \langle \{Q,\sigma (Q)\}\rangle$. $D$ is a line defined over $\mathbb {R}$. Since $A$ evinces $r_{\mathbb {C}}(P')$, the set $\langle \nu _d(\{Q,\sigma (Q)\}\rangle
\cap \langle \nu _d(F)\rangle$ is a unique point, $P_1$. Since $\sigma (A)=A$ and $\sigma (P)=P$, the uniqueness of $P_1$ implies
$\sigma (P_1) =P_1$. Since $P_1$ has real rank $\le d$ with respect to the rational normal curve $\nu _d(D)$ (\cite{co}, Proposition 2.1),
we get $r_{\mathbb {R}}(P') \le r_{\mathbb {R}}(P_1)+\sharp (F) \le d+b-2$. Hence $r_{\mathbb {R}}(P')=d+b-2$. Hence $d+b-2$ is a typical rank.
\end{proof}

\begin{remark}\label{ww2}
We may go further, up to the range $2d+b-4$ for all $m\ge 2$ if $d\gg \max \{m,b\}$, but the case $m=2$ is quite different and the case $m=3$ requires a different proof
from the case $m\ge 4$ (roughly speaking, after doing the case $m=3$ one may do the case $m>3$ by induction on $m$ as in steps (d1) and (d2) of the proof of Lemma \ref{ww1}).
\end{remark}

\section{The bivariate case}\label{S2}

In this section we prove Theorem \ref{u1} by proving a stronger result, which shows that bivariate polynomials with real rank $d$ are ubiquitous
and ``~typical~'' even in very small pieces of $\sigma _b(X_{1,d}(\mathbb {C}))(\mathbb {R})$.

Fix an integer $b$ such that $2 \le b \le (d+2)/2$. We need to prove that $d$ is a typical rank of $\sigma _b(X_{1,d}(\mathbb {C}))(\mathbb {R})$. We prove
the following stronger result.

\begin{theorem}\label{u2}
Fix integers $d, b, a$ such that $b \ge 2$, $d \ge 2b-1$ and $2\le 2a \le b$. Fix $Q_1,\dots ,Q_a\in \mathbb {P}^1(\mathbb {C})$ and $P_i\in \mathbb {P}^1(\mathbb {R})$, $1\le i \le b-2a$. Set
$$A:= \langle \{P_1,\dots ,P_{b-2a},Q_1,\sigma (Q_1),\dots ,Q_a,\sigma (Q_a)\}\rangle.$$Assume $\sharp (A)=b$ and set $M:= \langle \nu _d(A)\rangle (\mathbb {R})$.
Then the real projective space $M$ has dimension $b-1$ and there is a non-empty open subset $U\subset M$ for the euclidean topology such that $r_{\mathbb {R}}(P)=d$
for all $P\in U$.
\end{theorem}

\begin{proof}Since $d\ge b-1$, the real projective space $M$ has dimension $b-1$. Choose real homogeneous coordinates on $\mathbb {P}^1(\mathbb {C})$ so that $Q_j = \alpha _j\in \mathbb {C}\setminus \mathbb  {R}$ and $P_h = \beta_h\in \mathbb {R}$. The space $M$ parametrizes (up to a non-zero real multiplicative constant) all
degree $d$ homogeneous polynomial $f$ of the form
\begin{equation}\label{equ1}
\sum _{j=1}^{a} (c_j(z-\alpha _j)^d+\overline{c_j}(z-\overline{\alpha _j})^d) + \sum _{h=1}^{b-2a} d_h(z-\beta _h)^d \ \ c_j\in \mathbb {C}, \ d_h\in \mathbb {R}
\end{equation}
Set $D:= \langle \nu _d(\{Q_1,\sigma (Q_1)\}\rangle \subset \mathbb {P}^d(\mathbb {C})$. $D$ is a line defined over $\mathbb {R}$ which intersects the rational normal curve $X_{1,d}(\mathbb {C})$
at two distinct complex conjugate points. Fix any $O\in D(\mathbb {R})$. Up to a real change of coordinates we may assume $\alpha _1 = i$ and hence $\overline{\alpha }_1 =-i$. Hence the polynomial $f_0$ associated to $O$ is of the form
$$ f_O = c(z-i)^d +\overline{c}(z-i)^d, \ c\ne 0$$
We claim that $r_{\mathbb {R}}(O)=d$. Indeed, $r_{\mathbb {R}}(O)\le d$ by \cite{co}, Proposition 2.1. Since $O\notin \{Q_1,\sigma (Q_1)\}$, we have $r_{\mathbb {C}}(O)=2$.
As in the Proof of the Claim in the proof of Theorem \ref{w2} we see that $r_{\mathbb {R}}(O) \ge d$. Hence $r_{\mathbb {R}}(O)=d$.
Set $w:= (z+i)/(z-i)$. Since $w^d = -\overline{c}/c$ has $d$ distinct roots, we get that $f_O$ has $d$ distinct roots. It is easy to check that these roots are
real, but this is also a consequence of \cite{cr}, Corollary 1. Since $f_O$ has $d$ distinct real roots, there is a $\epsilon \in \mathbb {R}$, $\epsilon >0$, such
that any $f$ in (\ref{equ1}) has $d$ distinct real roots  if $c_1 =c$, $\vert c_j\vert < \epsilon $ for all $j=2,\dots ,a$ and $\vert d_h\vert < \epsilon$ for all $h=1,\dots ,b-2a$.
Varying $c$, we get a non-empty open subset $U$ of $M$ formed by real polynomials with $d$ distinct real roots. Hence $f\in U$ has real rank $d$
by \cite{cr}, Corollary 2.1.
\end{proof}

\begin{remark}\label{u3}
Take $A$ as in the statement of Theorem \ref{u2}. We will say that $A$ has type $(b-2a,a)$. We only assumed that $a\ge 1$. With this assumption we get
that $d$ is a typical rank for the corresponding real projective space $M$.
\end{remark}

\vspace{0.3cm}

\qquad {\emph {Proof of Theorem \ref{u1}.}}  A Zariski open subset of $\sigma _b(X_{1,d}(\mathbb {C}))$ is given by the union
of all sets $\langle \nu _d(B)\rangle \setminus (\cup _{B'\subsetneq B} \langle \nu _d(B')\rangle )$ with $B\subset \mathbb {P}^1(\mathbb {C})$ and $\sharp (B)=b$. A Zariski open and non-empty open subset
of $\sigma _b(X_{1,d}(\mathbb {C}))(\mathbb {R})$ is obtained taking the union over all $B\subset \mathbb {P}^1(\mathbb {C})$ such that $\sharp (B)=b$ and $\sigma (B)=B$.
Only the ones of type $(b,0)$ are not covered by Theorem \ref{u2}. Of course, if $B$ as type $(b,0)$, then $r_{\mathbb {R}}(P)=b$
for all $P\in \langle \nu _d(B)\rangle \setminus (\cup _{B'\subseteq B} \langle \nu _d(B')\rangle )$.\qed

\begin{acknowledgements}
We thank the referees for useful remarks.
\end{acknowledgements}

\end{document}